\documentclass[12pt, twoside, leqno]{article}

\usepackage{amsmath,amsthm}
\usepackage{amssymb}

\usepackage{enumerate}

\usepackage[T1]{fontenc}

\newtheorem{thm}{Theorem}[section]

\newtheorem{lem}[thm]{Lemma}
\newtheorem{prop}[thm]{Proposition}
\newtheorem{prob}[thm]{Problem}

\theoremstyle{definition}
\newtheorem{defin}[thm]{Definition}

\numberwithin{equation}{section}

\begin{document}


\baselineskip=17pt


\title{Construction of functions with the given cluster sets}

\author{Oleksandr V. Maslyuchenko\\
Instytute of Mathematics\\
Academia Pomeraniensis in S{\l}upsk\\
76-200 S{\l}upsk, Poland\\
E-mail: ovmasl@apsl.edu.pl
\and
Denys P. Onypa\\
Department of Mathematical
Analysis\\
Chernivtsi National University\\
58012 Chernivtsi, Ukraine\\
E-mail: denys.onypa@gmail.com}

\date{}

\maketitle


\renewcommand{\thefootnote}{}



\renewcommand{\thefootnote}{\arabic{footnote}}
\setcounter{footnote}{0}

\sloppy


\begin{abstract}
In this paper we continue our research of functions on the boundary of their domain and obtain some results on cluster sets of functions between topological spaces. In particular,
we prove that for a metrizable topological space $X$, a dense subspace $Y$ of a metrizable compact space $\overline{Y}$, a closed nowhere dense subset $L$ of $X$, an upper continuous compact-valued multifunction ${\Phi:L\multimap \overline{Y}}$ and a set $D\subseteq X\setminus L$ such that $L\subseteq \overline{D}$, there exists a function $f:D\to Y$ such that the cluster set $\overline{f} (x)$ is equal to $\Phi (x)$ for any $x\in L$.
\end{abstract}

\section{Introduction}

The notion of a cluster set was first formulated by Painleve \cite{P} in the study of analytic functions of a complex variable. A fundamental research of cluster sets of analytic functions was carried out in \cite{CL}. Actually, the notion of cluster sets is a topological notion which characterizes a behavior of a function on the boundary of the domain of this function. The oscillation \cite{KO} is another limit characteristic of functions which is tightly connected with the notion of cluster sets. Namely, the oscillation is the diameter of the cluster set.

The problem of the construction of functions with the given oscillation was investigated in \cite{KO,EP,DGP,KOW,MM,M1,M2,M3}. In \cite{MASON1} we researched the question on the limit oscillation of locally stable functions defined on open subsets of reals.
In \cite{MASON2} we proved that for any upper semicontinuous function ${f:F\to [0;+\infty]}$, which is defined on
the boundary $F=\overline{G}\setminus G$ of some open subset $G$ of a metrizable space $X$, there is a continuous
function $g:G\to \mathbb R$ such that the  oscillation $\omega_g(x)$ is equal to $f(x)$ for $x\in F$.
In this paper we deal with the question about the construction of functions with the given cluster sets.

Let $F$ be a multifunction from $X$ to $Y$. Recall, that \emph{the domain of} $F$ is
$$
\mathrm{dom}\, F=\{x\in X:F(x)\ne\emptyset\}.
$$
In the case when $\mathrm{dom}\, F=D$ we write $F:D\multimap Y$.
\emph{The graph} ${\rm gr} F$ of $F$ is a subset of $X\times Y$ which is defined by
$$\textstyle {\rm gr} F=\bigcup\limits_{x\in X} \{x\}\times F(x).$$
So, we have that $\mathrm{dom}\,F$ is the projection of $\mathrm{gr}\,F$ on $X$.

\begin{defin}
Let $X$ and $Y$ be topological spaces, $D\subseteq X$ and $f:D \to Y$. \emph{The cluster set of $f$ at a point} $x\in X$ is called

$$\textstyle \overline f (x)=\bigcap\limits_{U\in \mathcal{U}_x}\overline {f(U\cap D)},$$
where $\mathcal{U}_x$ denotes the collection of all neighborhoods of $x$ in $X$.
The corresponding multifunction $\overline{f}$ we call \emph{the cluster multifunction of} $f$.
\end{defin}

It is not hard to show that the graph $\mathrm{gr}\, \overline{f}$ of the cluster multifunction $\overline{f}$ is the closure of the graph $\mathrm{gr}\, f$ of $f$.
In general, $D={\rm dom} f \subseteq {\rm dom} \overline f \subseteq \overline D$.
Let $Y$ be a compact space. By the Kuratowski theorem \cite{EN} the projection on $X$
of any closed subset of $X\times Y$ is closed in $X$. So, the domain of $\overline{f}$ is a closed subset of $X$  as the projection of $\mathrm{gr}\, \overline{f}=\overline{\mathrm{gr}\,f}$. Then $\mathrm{dom}\, \overline{f}=\overline{D}$.
It is easy to prove that  $f$ is continuous if and only if $\overline f(x)=\{f(x)\}$ for any $x\in D$. Moreover, if $Y$ is a metric compact, then the oscillation $\omega_f(x)$ of $f$ at a point $x\in \overline{D}$ is equal to the diameter $\mathrm{diam}\, \overline{f}(x)$ of the cluster set $\overline{f}(x)$.

So, we deal with the following problem.

\begin{prob} Let $X$  be a topological space and $Y$ be a dense subspace of a compact space $\overline{Y}$, $D\subseteq X$ and $L\subseteq\overline D\setminus D$. Describe all multifunctions $\Phi:L\multimap \overline{Y}$ such that there exists a function $f:D\rightarrow Y$ with $\overline f(x)= \Phi(x)$ for any $x\in L$.
\end{prob}

Proposition~\ref{P6} and Theorem~\ref{T2} give an answer on this problem in the case when
$X$ and $\overline{Y}$ are metrizable spaces and
 $L$ is a closed nowhere dense subset of $X$.

\section{$\sigma$-discrete sets}

\begin{defin}
A subset $S$ of a space $X$ is called
\begin{itemize}

\item \emph{discrete} if for any $x\in S$ there exists a neighborhood $U$ of $x$ in $X$ such that $U\cap S=\{x\}$;

\item an $\varepsilon$-\emph{net} if for any $x\in X$ there exists $y\in S$ such that $d(x,y)< \epsilon$;

\item $\varepsilon$-\emph{separated} if ${|s-t|_X\geq\varepsilon}$ for any different points $s,t\in S$;

\item \emph{separated} if it is $\varepsilon$-separated for some $\varepsilon>0$;

\item $\sigma$-\emph{discrete} if there exists a sequence of discrete sets $S_n$ such that ${S=\bigcup\limits_{n=1}^{\infty} S_n}$.

\end{itemize}

\end{defin}
\begin{prop}\label{P1}
Let $X$ be a metrizable space and $\varepsilon >0$. Then there exists an $\varepsilon$-separated $\varepsilon$-net $S$ in $X$.
\end{prop}

\begin{proof}[Proof]
Since the $\varepsilon$-separability is a property of finite character, using the Teichm\"uller-Tukey lemma \cite[p.8]{EN} we build a maximal $\varepsilon$-separated set $S$. Then the maximality of $S$ implies that $S$ is an $\varepsilon$-net.
\end{proof}

\begin{prop}\label{P2}
Let $X$ be a metrizable space. Then there exist separated sets $A_n$ such that $A=\bigcup\limits_{n=1}^{\infty} A_n$ is a dense subset of $X$. In particular, the set $A$ is $\sigma$-discrete and dense in $X$.
\end{prop}

\begin{proof}[Proof]
Using Proposition~\ref{P1} we pick an $\frac{1}{n}$-separated $\frac{1}{n}$-net $A_n$ for any $n$ and put $A=\bigcup\limits_{n=1}^{\infty} A_n$. Obviously, the set $A$ is dense. Since every separated set is discrete, then $A$ is $\sigma$-discrete as a union of separated sets.
\end{proof}

\begin{prop}\label{P3}
Let $X$ be a metrizable space and $S\subseteq X$. Then the following statements are equivalent:

$(i)$ $S$ is $\sigma$-discrete;

$(ii)$ there exists a sequence of separated sets $S_n$ such that $S=\bigsqcup\limits_{n=1}^{\infty} S_n$;

$(iii)$ for any infinitesimal sequence of positive numbers $\varepsilon_n$ there exists a sequence of sets $S_n$ such that $S_n$ is $\varepsilon_n$-separated for any $n$ and $S=\bigsqcup\limits_{n=1}^{\infty} S_n$;

$(iv)$ there exists a sequence of closed discrete sets $S_n$ in $X$ such that $S=\bigcup\limits_{n=1}^{\infty} S_n$.
\end{prop}

\begin{proof}[Proof.]
Prove that $(i)$ implies $(ii)$. Since $S$ is $\sigma$-discrete, there is a sequence of discrete sets $A_n$ such that $S=\bigcup\limits_{n=1}^{\infty} A_n$. Using Proposition~\ref{P2} we pick separated sets $A_{n,m}$ such that $\bigcup\limits_{m=1}^{\infty} A_{n,m}$ is dense in $A_n$. As $A_n$ is a discrete set, so every dense set in it equals $A_n$.
Therefore, we have that $A_n=\bigcup\limits_{m=1}^{\infty} A_{n,m}$. Renumber ${\mathbb N}^2=\{(n_k,m_k):k\in \mathbb N\}$. Put $B_k=A_{n_k,m_k}$ and $S_k=B_k\setminus \bigcup\limits_{j<k} B_j$ for any $k\in \mathbb N$. The sequence $(S_k)$ is desired.

Prove that $(ii)$ implies $(iii)$. Let $S=\bigsqcup\limits_{n=1}^{\infty} T_n$, where $T_n$ is $\delta_n$-separated. Fix an infinitesimal sequence of positive numbers $\varepsilon_n$. Since $\varepsilon_n\to 0$, there is $n_1$ such that $\varepsilon_{n_1}<\delta_1$. Analogically there is $n_2>n_1$ such that $\varepsilon_{n_2}<\delta_2$. Continuing the process we have that for any $k\in \mathbb N$ there is $n_k>n_{k-1}$ such that $\varepsilon_{n_k}<\delta_k$. Thus, the set $T_k$ is $\varepsilon_{n_k}$-separated.
Finally, put $S_n=T_k$ if $n=n_k$ for some $k\in \mathbb N$ and $S_n=\emptyset$ otherwise.

Obviously, $(iii)$ implies $(iv)$, because every separated set is closed.

The implication $(iv)\Rightarrow(i)$ is clear.
\end{proof}

\begin{prop}\label{P4}
Let $X$ and $Y$ be metrizable spaces, $Y$ be compact, $M\subseteq X\times Y$ be a $\sigma$-discrete set. Then a projection pr$_X M$ is $\sigma$-discrete.
\end{prop}

\begin{proof}[Proof.]
According to Proposition~\ref{P3} we have that every $\sigma$-discrete set is a union of closed discrete sets. Thus, it is enough to show that projection of closed discrete set is closed and discrete.
Let $M$ be a closed discrete subset of $X$. Show that pr$_X M=S$ is closed and discrete. Let $M'$ be some subset of $M$. Since $M$ is discrete and closed, the set $M'$ is closed in $X\times Y$. By the Kuratowski theorem \cite[Theorem 3.1.16]{EN} we conclude that $S'={\rm pr}_X M'$ is closed in $X$. Therefore, all subsets of $S$ is closed in $X$. Thus, we have that $S$ is closed and discrete.
\end{proof}

\section{Cluster sets of a function with a discrete domain}
For a metric space $(X,d)$, a point $a\in X$, a nonempty set $A\subseteq X$ and $\varepsilon>0$ we denote ${d(a,A)=\inf\limits_{x\in A} d(a,x)}$ and
$$B(a,\varepsilon)=\{x\in X:d(x,a)<\varepsilon\},\qquad B(A,\varepsilon)=\{x\in X:d(x,A)<\varepsilon\},$$
$$B[a,\varepsilon]=\{x\in X:d(x,a)\leq\varepsilon\},\qquad B[A,\varepsilon]=\{x\in X:d(x,A)\leq\varepsilon\}.$$
We also use more exact notations: $B_X(a,\varepsilon)$, $B_X(A,\varepsilon)$, $B_X[a,\varepsilon]$, $B_X[A,\varepsilon]$.
\begin{prop}\label{P5}
Let $X$ be a metrizable space, $F$ be a nonempty closed subset of $X$, $A$ be separable subset of $X$ such that ${F\subseteq \overline{A}}$. Then there exist a sequence of points $x_n\in A$ such that $F$ is equal to the set $\bigcap\limits_{n=1}^{\infty}\overline{\{x_m:m\geq n\}}$ of all limit points of $(x_n)$.

\end{prop}

\begin{proof}[Proof.]
Without loss of generality we may assume that $X=\overline{A}$. So, $X$ is a separable metrizable space. Then by \cite[Theorem 4.3.5]{EN} we can choose a totally bounded metric on $X$.
Let $S_k$ be a finite $\frac{1}{2k}$-net in $X$. Pick ${s\in S_k}$. Since $A$ is dense in $X$, there is $a_s \in A$ such that $d(x,a_s)<\frac{1}{2k}$. Put ${A_k=\{a_s:s\in S_k\}}$. Then $A_k$ is a finite $\frac{1}{k}$-net in $X$. Consider finite sets $B_k=A_k\cap B(F,\frac{1}{k})$. Let
$$B_1=\{x_1,\dots,x_{n_1}\}, B_2=\{x_{n_1+1},\dots, x_{n_2}\}, \dots, B_k=\{x_{n_{k-1}+1}, \dots, x_{n_k}\}, \dots$$
Show that the sequence $(x_n)$ is desired.

Since $\{x_m:m>n_k\}=\bigcup\limits_{j>k} B_j\subseteq B(F,\frac{1}{k})$, we have that ${\overline{\{x_m:m>n_k\}}\subseteq B[F,\frac{1}{k}]}$. So,
$$ \bigcap\limits_{m=1}^{\infty}\overline{\{x_n:n\geq m\}}=\bigcap\limits_{k=1}^{\infty}\{x_n:n\geq n_k\}\subseteq \bigcap\limits_{k=1}^{\infty} B[F,\textstyle\frac{1}{k}]=F.$$

Let $x\in F$. Then there exists $y_k\in A_k$ such that $d(x,y_k)<\frac{1}{k}$. Thus, $y_k\in A_k\cap B(F,\frac{1}{k})=B_k$. Then there exists $m_k$ such that $n_{k-1}<m_k\leq n_k$ and $y_k=x_{m_k}$. Therefore, $(y_k)$ is a subsequence of $(x_n)$ and $y_k \to x$. So, $x$ is a limit point of $x_n$.
\end{proof}

Recall, that a multifunction $\Phi:X\multimap Y$ is called \emph{upper continuous} if for any point $x\in X$ and for any open set $V$ with $\Phi(x)\subseteq V$ there exists a neighborhood $U$ of $x$ in $X$ with $\Phi(U)\subseteq V$.

\begin{prop}\label{P6}
Let $X$ be a topological space, $Y$ be a compact, $D\subseteq X$ and $f:D\to Y$. Then $\overline f:\overline D\multimap Y$ is an upper continuous compact-valued multifunction.
\end{prop}

\begin{proof}[Proof.]
For any $x\in \overline{D}$ we have that $\overline{f}(x)=\bigcap\limits_{U\in \mathcal{U}_x}\overline {f(U)}$. So, $\overline{f}(x)$ is a closed subset of a compact space $Y$. Therefore, $\overline{f}(x)$ is compact.

Prove that $\overline f$ is upper continuous. Let $x_0\in \overline D$ and let $V$ be an open neighborhood of $\overline f(x_0)$.
Put $\mathcal{F}=\{\overline {f(U\cap D)}: U \mbox{ is a neighborhood of } x_0\}$.
Then $\overline f(x_0)=\bigcap \mathcal{F}\subseteq V$. Since $Y$ is compact, there exist $F_1,F_2,\dots,F_n\in \mathcal{F}$ such that $\bigcap\limits_{k=1}^{n}F_k\subseteq V$. Choose $U_k\in \mathcal{U}_{x_0}$ such that $F_k=\overline {f(U_k)}$. Put $U_0={\rm int} \bigcap\limits_{k=1}^{n} U_k$. Then $U_0$ is a neighborhood of $x_0$ and $\overline{f(U_0\cap D)}\subseteq \bigcap\limits_{k=1}^{n}{\overline {f(U_k)}}=\bigcap\limits_{k=1}^{n} F_k\subseteq V$.
Therefore, for any $x\in U_0\cap \overline D$ we have that $\overline{f}(x)=\bigcap\limits_{U\in \mathcal{U}_x}\overline {f(U)}\subseteq \overline{f(U_0)}\subseteq V$. So, $\overline{f}$ is upper continuous at $x_0$.
\end{proof}

\begin{thm}\label{T1}
Let $X$ be a metrizable space, $Y$ be a dense subset of a metrizable compact $\overline{Y}$, $L$ be a closed nowhere dense subset of $X$, $D=X\setminus L$ and $\Phi:L \multimap \overline{Y}$ be an upper continuous compact-valued multifunction. Then there is a set $A$ such that $A$ is discrete in $D$, $\overline{A}\setminus A=L$, and there exists a function $f:A\to Y$, such that $\overline{f}(x)=\Phi(x)$ for any $x\in L$.
\end{thm}

\begin{proof}[Proof.]
Using Proposition~\ref{P2} for a metrizable space ${\rm gr} F$ we conclude that there exists a $\sigma$-discrete set $M$, which is dense in ${\rm gr} F$. Then $M$ is $\sigma$-discrete in $X\times \overline{Y}$ too. By Proposition~\ref{P4} we have that the projection $S={\rm pr}_X M$ is $\sigma$-discrete in $X$. Furthermore,
$$L={\rm pr}_X {\rm gr} \Phi \subseteq {\rm pr}_X \overline{M} \subseteq \overline{{\rm pr}_X M}=\overline{S}.$$
Then $S$ is a dense subset of $L$. By Theorem~\ref{P3} there exists a sequence of $\frac{3}{n}$-separated sets $S_n$ such that $S=\bigsqcup\limits_{n=1}^{\infty} S_n$. Set $T_n=\bigsqcup\limits_{m<n} S_m$ for any $n\in \mathbb{N}$. Now we construct families of points $x_{s,k}$ and $y_{s,k}$ where $s\in S$, $k\in \mathbb{N}$, satisfying the following conditions:
$$x_{s,k}\in D \mbox{ for any } s\in S \mbox{ and } k\in \mathbb{N}; \leqno (1)$$
$$\textstyle d(x_{s,k},s)<\frac{1}{n} \mbox{ for any } n\in \mathbb{N}, s\in S_n \mbox{ and } k\in \mathbb{N}; \leqno (2)$$
$$x_{s,k}\to s \mbox{ as } k\to \infty \mbox{ for any } s\in S; \leqno (3)$$
$$x_{s,k}\neq x_{t,j} \mbox{ if } (s,k)\neq (t,j); \leqno (4)$$
$$y_{s,k}\in Y \mbox{ for any } s\in S \mbox{ and } k\in \mathbb{N}; \leqno (5)$$
$$\Phi (s) \mbox{ is the set of limit points of sequence } {(y_{s,k})}_{k=1}^{\infty} \mbox{ for any } s\in S; \leqno (6)$$
$$d(y_{s,k},\Phi(s))<\textstyle\frac{1}{n} \mbox{ for any } n\in \mathbb N, s\in S_n \mbox{ and } k\in \mathbb{N}. \leqno (7)$$
Fix $s\in S_1$ and put $D_s=B(s,1)\cap D$. Since $\overline{D}=X$, $s\in \overline{D}_s$. Thus, $s\in \overline{D}_s\setminus D_s$. Therefore, there exists a sequence of different points $x_{s,k}\in D_s$, that tends to $s$. So, we have that conditions (1) -- (3) hold. Since a family of balls $(B(s,1))_{s\in S_1}$ is discrete, then a family of sets $(D_s)_{s\in S_1}$ is discrete too. Thus, the condition (4) holds.

Assume that points $x_{s,k}$ are constructed for some number $n>1$ and for any $s\in T_n$ and $k\in \mathbb{N}$, so that conditions (1) -- (4) hold. Fix $s\in S_n$ and put
$$\textstyle D_s=B(s,\frac{1}{n})\cap D\setminus \{x_{t,j}:t\in T_n, j\in \mathbb N\}.$$
Prove that $s\in \overline{D}_s$. Let $A_t=\{x_{t,j}:j\in \mathbb{N}\}$ for any $t\in T_n$. Since $x_{t,j}\to t$, $\overline{A}_t=A_t\cup \{t\}$. So $s\notin \overline{A}_t$ for any $t\in T_n$. For any $m<n$ a family of balls $(B(t,\frac{1}{m}))_{t\in S_m}$ is discrete. Besides, by the condition (2) we have that $A_t\subseteq B(t,\frac{1}{m})$ if $m<n$ and $t\in S_m$. Then a family $(A_t)_{t\in S_m}$ is discrete as $m<n$. Since $T_n=\bigsqcup\limits_{m<n} S_m$, we have that a family $(A_t)_{t\in T_n}$ is locally finite. In this case, $\overline{\bigcup\limits_{t\in T_n} A_t}=\bigcup\limits_{t\in T_n} \overline{A}_t\not \ni s$.

So, a set
$$\textstyle U=X\setminus\bigcup\limits_{t\in T_n} A_t=X\setminus \{x_{t,j}:t\in T_n, j\in \mathbb N\}$$
is a neighborhood of $s$. Since $\overline{D}=X$, $s\in \overline{B(s,\frac{1}{n})\cap D}$. Thus, ${s\in \overline{B(s,\frac{1}{n})\cap D\cap U}=\overline{D}_s}$. Next, taking into account that $s\notin D_s$, we have that $s\in \overline{D}_s\setminus D_s$. Therefore, there exists a sequence of different points ${x_{s,k}\in D_s}$ that tends to $s$. So, the conditions (1) -- (3) hold. Besides, ${x_{s,k}\neq x_{t,j}}$ as $t\in T_n$, $k,j\in \mathbb N$. Since a family of balls $(B(s,\frac{1}{n}))_{s\in S_n}$ is discrete, we have that  a family $(D_s)_{s\in S_n}$ is discrete too. Therefore, the condition (4) holds.

Construct points $y_{s,k}$. Fix $n\in \mathbb N$ and $s\in S_n$. A set $B_{\overline{Y}}(\Phi(s),\frac1n)$ is an open neighborhood of a set $\Phi(s)$ in $\overline{Y}$ and $Y$ is dense in $\overline{Y}$. Then for a set
$Y_s=B_{Y}(\Phi(s),\frac1n)$
we have that $Y_s=B_{\overline{Y}}(\Phi(s),\frac1n)\cap Y$. Then ${\Phi(s)\subseteq \overline{Y}_s}$. Since every metrizable compact is a hereditarily separable, we have that  the set $Y_s$ is separable. Using Proposition~\ref{P5} with $F=\Phi(s)$ and $A=Y_s$ we conclude that there exists a sequence of points $y_{s,k}\in Y_s$ such that $\Phi(s)$ is the set of limit points of $(y_{s,k})_{k=1}^{\infty}$. Therefore, the conditions (5)~--~(7) hold.

Put $A=\{x_{s,k}:s\in S,k\in \mathbb N\}$ and check that $\overline{A}\setminus A=L$. Firstly, by (3) we have that $S\subseteq \overline{A}$. Thus, $L=\overline S\subseteq \overline{A}$. Secondly, by (1) we have that $L\cap A=\emptyset$. Thus, $L\subseteq \overline{A}\setminus A$.
Prove inverse inclusion. For any $s\in S$ put $A_s=\{x_{s,k}:k\in \mathbb N\}$.
By (3) we have that $\overline{A}_s=A_s\cup \{s\}$. Since a family of balls $(B(s,\frac{1}{n}))_{s\in S_n}$ is discrete, a family $(A_s)_{s\in S_n}$ is discrete too. So, $\overline{\bigcup\limits_{s\in S_n} A_s}=\bigcup\limits_{s\in S_n} \overline{A_s}=\bigcup\limits_{s\in S_n}(A_s\cup \{s\})=S_n\cup (\bigcup\limits_{s\in S_n} A_s)$ for every number $n$. Besides, by (2) we have that for any $m\geq n$ and $s\in S_m$ we have an inclusion $A_s\subseteq B(s,\frac{1}{m})\subseteq B[L,\frac{1}{n}]$.
Hence, $\overline{\bigcup\limits_{m\geq n}\bigcup\limits_{s\in S_m} A_s}\subseteq B[L,\frac{1}{n}]$ for any number $n$. So,

$$\textstyle \overline{A}=\overline{\bigcup\limits_{m<n}\bigcup\limits_{s\in S_m} A_s \cup \bigcup\limits_{m\geq n}\bigcup\limits_{s\in S_m} A_s}=\bigcup\limits_{m<n}\overline{\bigcup\limits_{s\in S_m} A_s} \cup \overline{\bigcup\limits_{m\geq n}\bigcup\limits_{s\in S_m} A_s}$$
$$\textstyle \subseteq \bigcup\limits_{m<n}(S_m\cup\bigcup\limits_{s\in S_m} A_s)\cup B[L,\frac{1}{n}] \subseteq B[L,\frac{1}{n}] \cup A.$$
Therefore, $\overline{A}\setminus A\subseteq B[L,\frac{1}{n}]$ for any $n$. Hence, $\overline{A}\setminus A\subseteq \bigcap\limits_{n=1}^{\infty} B[L,\frac{1}{n}]=L$.

Define a function $f:A\to Y$ by
$$f(x_{s,k})=y_{s,k} \mbox{ for }  s\in S \mbox{ and } k\in \mathbb N$$
and show that $f$ is desired. Fix $s_0\in L$ and find out that $\overline{f}(s_0)=\Phi(s_0)$.

Prove that $\overline{f}(s_0)\subseteq \Phi(s_0)$. Pick $\varepsilon>0$ and show that $\overline{f}(s_0)\subseteq B(\Phi(s_0),2\varepsilon)$. Since $\Phi$ is upper continuous, then there exists a neighborhood $U_0$ of $s_0\in X$ such that
$$\Phi(U_0\cap L)\subseteq B(\Phi(s_0),\varepsilon). \leqno (8)$$
Choose $\delta>0$ such that $B(s_0,2\delta)\subseteq U_0$ and find a number $n_0$ such that $\frac{1}{n_0}<\delta$ and $\frac{1}{n_0}<\varepsilon$. Put $U_1=B(s_0,\delta)$ and prove that
$$d(f(x),\Phi(s_0))<2\varepsilon \mbox{ as } x\in (A\setminus T_{n_0})\cap U_1. \leqno (9)$$
Let $x\in A\setminus T_{n_0}=\bigsqcup\limits_{n\geq n_0} S_{n}$ such that $x\in U_1$. Then $x=x_{s,k}$ for some $n\geq n_0$, $k\in \mathbb N$ and $s\in S_n$. Since $x_{s,k}\in U_1$, then ${d(x_{s,k},s_0)<\delta}$. Next, by (2) we have that $d(x_{s,k},s)<\frac{1}{n}\leq\frac{1}{n_0}<\delta$. Thus, ${d(s_0,s)<2\delta}$. So, ${s\in B(s_0,2\delta)\subseteq U_0}$. Therefore, $s\in U_0\cap L$. By (8) we have that $\Phi(s)\subseteq B(\Phi(s_0),\varepsilon)$. But $f(x)=y_{s,k}$. By (7) we have that ${d(f(x),\Phi(s))=d(y_{s,k},\Phi(s))<\frac{1}{n}\leq \frac{1}{n_0}<\varepsilon}$.

Thus, there is $y\in \Phi(s)$ such that $d(f(x),y))<\varepsilon$. But $d(y,\Phi(s_0))<\varepsilon$, because $y\in \Phi(s)\subseteq B(\Phi(x_0),\varepsilon)$. Therefore, $d(f(x),\Phi(s_0))<2\varepsilon$. So, (9) holds.

As was remarked previously, $(A_s)_{s\in T_{n_0}}$ is a locally finite. Besides, by (3) we have that $\overline{A}_s=A_s\cup \{s\}$. Thus, by (4) we have that $(\overline{A}_s)_{s\in T_{n_0}}$ is disjoint. So, we have that $(A_s)_{s\in T_{n_0}}$ is discrete.

Let us show that there is a neighborhood $U_2\subseteq U_1$ of a point $s_0\in X$ such that
$$d(f(x),\Phi(s_0))< 2\varepsilon \mbox{ as } x\in T_{n_0}\cap U_2. \leqno (10)$$
Firstly, consider the case when $s_0\notin T_{n_0}$. Then for any $t\in T_{n_0}$ we have that $s_0\notin \{t\}\cup A_t=\overline{A}_t$. Thus, $s_0\notin \bigcup\limits_{t\in T_{n_0}}\overline{A}_t=\overline{\bigcup\limits_{t\in T_{n_0}} A_t}$.
Then there is a neighborhood $U_2$ of $s_0$ such that $U_2\cap A_t=\emptyset$ as $t\in T_{n_0}$. Therefore, the condition (10) holds.

Now consider the case when $s_0\in T_{n_0}$. Since $s_0\in A_{s_0}\cup \{s_0\}=\overline{A}_{s_0}$ and a family $(A_s)_{s\in T_{n_0}}$ is discrete, we have that there is a neighborhood $U_3\subseteq U_1$ of a point $s_0$ such that $U_3\cap A_s=\emptyset$ as $s\in T_{n_0}\setminus \{s_0\}$. Next, by (6) we have that $\Phi (s_0)$ is a set of limit points of a sequence $\mathop{(y_{s_0,k})}_{k=1}^{\infty}$. Thus, there is $k_0\in \mathbb N$ such that
$$d(y_{s_0,k},\Phi(s_0))<2\varepsilon \mbox{ for any } k\geq k_0. \leqno (11)$$
Put $U_2=U_3\setminus \{x_{s_0,k}:k<k_0\}$. Then we have that ${U_2\cap T_{n_0}\subseteq \{x_{s_0,k}:k\geq k_0\}}$. Thus, (10) implies (11).

Therefore, there is a neighborhood $U_2\subseteq U_1$ of $s_0$ such that a condition (10) hold. Thus, by (9) and (10) we have that $d(f(x),\Phi(s_0))<2\varepsilon$ as $x\in A\cap U_2$. So $f(A\cap U_2)\subseteq B(\Phi(s_0),2\varepsilon)$. Hence, we have that
$$\textstyle \overline{f}(s_0)=\bigcap\limits_{U\in \mathcal{U}_{s_0}}\overline{f(A\cap U)}\subseteq \overline{f(A\cap U_2)}\subseteq B[\Phi(s_0),2\varepsilon]$$
for any $\varepsilon>0$. Since $\Phi(s_0)$ is closed, we have that
$$\textstyle \overline{f}(s_0)\subseteq \bigcap\limits_{\varepsilon>0}B[\Phi(s_0),2\varepsilon]=\Phi(s_0).$$

Let us find out that $\Phi(s_0)\subseteq \overline{f}(s_0)$. It is sufficient to show that ${\Phi(s_0)\subseteq \overline{f(U\cap A)}}$ for any open neighborhood $U$ of $s_0$. Fix some open neighborhood $U_0$ of $s_0$ and a point $y_0\in \Phi(s_0)$. Show that $y_0\in \overline{f(U_0\cap A)}$. Pick an open neighborhood $V_0$ of $y_0$ and show that $V_0\cap f(U_0\cap A)\neq \emptyset$. Since $(s_0,y_0)\in {\rm gr} \Phi=M $, $(U_0\times V_0)\cap M\neq \emptyset$. Thus, there is a point $(s,y)\in M \cap (U_0\times V_0)$. Then $s\in {\rm pr}_X M=S$, $s\in U_0$ and $y\in \Phi(s)\cap V_0$. Next, by (3) we have that there exists $k_0\in \mathbb N$ such that $x_{s,k}\in U_0$ as $k\geq k_0$. By (6) we have that $y$ is a limit point of a sequence  $\mathop{(y_{s,k})}_{k=1}^{\infty}$. Thus, there exists $k\geq k_0$ such that $y_{s,k}\in V_0$. So, we have that $x_{s,k}\in U_0\cap A$ and $f(x_{s,k})=y_{s,k}\in V_0\cap f(U_0\cap A)$. Hence, $V_0\cap f(U\cap A)\neq \emptyset$, and thus, $y_0\in \overline{f(U_0\cap A)}$.
\end{proof}

\section{The main result}
\begin{lem}\label{L1}
Let $X$ be a metrizable topological space, $Y$ be a dense subspace of a metrizable compact space $\overline{Y}$, $L$ be a closed nowhere dense subset of $X$, ${\Phi:L\multimap \overline{Y}}$ be an upper continuous compact-valued multifunction and $D\subseteq X\setminus L$ such that $L\subseteq \overline{D}$. Then there exists a function $f:D\to Y$ such that $\overline{f}(x)\subseteq \Phi(x)$.
\end{lem}

\begin{proof}[Proof.]
Let $x\in D$ and $r(x)=d(x,L)=\inf\limits_{y\in L} d(x,y)>0$. Then there exists $h(x)\in L$ such that $d(x,h(x))<2r(x)$. Since $Y$ is dense in $\overline{Y}$, there is a point $f(x)\in B_{\overline{Y}}(\Phi(h(x)),r(x))\cap Y$. Let us show that $f:D\to Y$ is desired.

Let $a\in L$, $\varepsilon>0$ and $V=B_{\overline{Y}}(\Phi(a),\frac{\varepsilon}{2})$. Since $\Phi$ is upper continuous, then there is $\delta>0$ such that $\delta<\varepsilon$ and for a neighborhood $U_0=B_X(a,\delta)$ we have that $\Phi(U_0\cap L)\subseteq V$. Let $U_1=B_X(a,\frac{\delta}{3})$. Pick $x\in U_1\bigcap D$. Then $r(x)=d(x,L)\leq d(x,a)<\frac{\delta}{3}$. Therefore, $d(x,h(x))<2r(x)<\frac{2\delta}{3}$. Thus, $d(a,h(x))<d(a,x)+d(x,h(x))<\frac{\delta}{3}+\frac{2\delta}{3}=\delta$. Hence, $h(x)\in U_0$. Besides, ${r(x)<\frac{\delta}{3}<\frac{\varepsilon}{3}<\frac{\varepsilon}{2}}$. Thus, $f(x)\in B_{\overline{Y}}(\Phi(h(x)),r(x))\subseteq B_{\overline{Y}}(\Phi(h(x)),\frac{\varepsilon}{2})$. Therefore, there exists $y\in \Phi(h(x))$ such that $d(f(x),y)<\frac{\varepsilon}{2}$. Next, $y\in \Phi(h(x))\subseteq \Phi(U_0\cap D)\subseteq V=B_{\overline{Y}}(\Phi(a),\frac{\varepsilon}{2})$. Then there exists $z\in \Phi(a)$ such that $d(y,z)<\frac{\varepsilon}{2}$. In this case, we have that $d(f(x),z)<d(f(x),y)+d(y,z)<\varepsilon$. Thus, $f(x)\in B_Y(\Phi(a),\varepsilon)$ for any $x\in U_1\cap D$. Hence, $\overline{f(U_1\cap D)}\subseteq \overline{B_{\overline{Y}}(\Phi(a),\varepsilon)}\subseteq B_{\overline{Y}}[\Phi(a),\varepsilon]$. Thus, $$\overline{f}(a)=\bigcap\{\overline{f(U\cap D)}: U\in \mathcal{U}_a\}\subseteq \overline{f(U_1\cap D)}\subseteq B_{\overline{Y}}[\Phi(a),\varepsilon]$$
for any $\varepsilon>0$. Therefore,
$\overline{f}(a)\subseteq \bigcap\limits_{\varepsilon>0}B_{\overline{Y}}[\Phi(a),\varepsilon]=\Phi(a)$,
because a set $\Phi(a)$ is closed.
\end{proof}

\begin{lem}\label{L2}
Let $(M,\leq)$ be a directed set, ${(A_m)}_{m\in M}$ and ${(B_m)}_{m\in M}$ be decreasing sequences of sets. Then $\bigcap\limits_{m\in M}(A_m\cup B_m)\subseteq (\bigcap\limits_{m\in M}A_m)\cup (\bigcap\limits_{m\in M}B_m)$.
\end{lem}

\begin{proof}[Proof.]
Let $x\in \bigcap\limits_{m\in M}(A_m\cup B_m)$. Then for any $m\in M$ we have that $x\in A_m\cup B_m$. That is $x\in A_m$ or $x\in B_m$. Let $K=\{k\in M:x\in A_k\}$ and $L=M\setminus K$. Then for any $l\in L$ we have that $x\in B_l$.

Assume that $K$ is cofinal. Then for any $m\in M$ there is $k\in K$ such that $m\leq k$. If $m\in M$ then $m\leq k$ for some $k\in M$ and thus, $x\in A_k\subseteq A_m$. Therefore, $x\in A_m$ for any $m$. Thus, $x\in \bigcap\limits_{m\in M} A_m$.

Let $K$ is not cofinal. Then
\[
\mbox{there exists } m_0\in M \mbox{ such that for any } k\in K \mbox{ we have that } m_0\nleq k. \eqno (*)
\]
Check that a set $L$ is cofinal. Pick $m\in M$. Since $M$ is a directed set, there exists $l\in M$ such that $l\geq m$ and $l\geq m_0$. Then by $(*)$ we have that $l\notin K$. That is $l\in M\setminus K=L$. So, $L$ is cofinal. If $m\in M$ then there exists $l\in L$ such that $m\leq l$ and thus, $x\in B_l\subseteq B_m$. Therefore, $x\in \bigcap\limits_{m\in M} B_m$.
\end{proof}

\begin{thm}\label{T2}
Let $X$ be a metrizable topological space, $Y$ be a dense subspace of a metrizable compact space $\overline{Y}$, $L$ be a closed nowhere dense subset of $X$, $\Phi:L\multimap \overline{Y}$ be an upper continuous compact-valued multifunction and $D\subseteq X\setminus L$ such $L\subseteq \overline{D}$. Then there exists a function $f:D\to Y$ such that $\overline{f} (x)=\Phi (x)$ for any $x\in L$.
\end{thm}

\begin{proof}[Proof.]
By Theorem~\ref{T1} we have that there exists a closed discrete set ${D_1\subseteq D}$ and a function $f_1:D_1\to Y$ such that $\overline{D_1}\setminus D_1=L$ and $\overline{f}_1 (x)=\Phi(x)$ as $x\in L$.

Put $D_2=D\setminus D_1$, $L_2=\overline{D}_2\cap L$ and $\Phi_2=\Phi|_{L_2}$. Obviously, that $L_2$ is closed and nowhere dense and $\Phi_2$ is an upper continuous compact-valued multifunction, moreover $L_2\subseteq \overline{D}_2$. Thus, by Lemma~\ref{L1} there exists a function $f_2:D_2\to Y$ such that $\overline{f}_2(x)\subseteq \Phi_2 (x)=\Phi (x)$ as $x\in L_2$. Define a function $f:D\to Y$ as follows:
$$f(x)=\left\{\begin{array}{ll}
f_1(x),\mbox{ if } x\in D_1 ,\\
f_2(x),\mbox{ if } x\in D_2;
\end{array}\right.$$
for any $x\in D=D_1\sqcup D_2$. Check that $f$ is desired. Fix $x_0\in D$ and show that $\overline{f}(x_0)=\Phi (x_0)$.

Firstly, let us find out that $\overline{f}(x_0)\supseteq \Phi(x_0)$. Indeed, taking into account that $x_0\in L\subseteq \overline{D}_1$ and $D_1\subseteq D$ we have that
$$\overline{f}(x_0)=\bigcap\{\overline{f(U\cap D)}: U\in \mathcal{U}_{x_0}\}\supseteq \bigcap\{\overline{f(U\cap D_1)}: U\in \mathcal{U}_{x_0}\}$$
$$ =\bigcap\{\overline{f_1(U\cap D_1)}: U\in \mathcal{U}_{x_0}\}=\overline{f}_1(x_0)=\Phi(x_0).$$
Check that $\overline{f}(x_0)\subseteq \Phi(x_0)$. Consider the case when $x_0\in L\setminus L_2$. In this case we have that $x_0\notin \overline{D}_2$. Thus, there exists a neighborhood $U_0$ of $x_0$ such that $U_0 \cap D_2=\emptyset$. Then for $U\subseteq U_0$ we have that $U\cap D=U\cap D_1$. Therefore,
$$ \overline{f}(x_0)=\bigcap\{\overline{f(U\cap D)}: U\in \mathcal{U}_{x_0}\}=\bigcap\{\overline{f(U\cap D)}: U\subseteq U_0, U\in \mathcal{U}_{x_0}\}$$
$$ = \bigcap\{\overline{f(U\cap D_1)}: U\subseteq U_0, U\in \mathcal{U}_{x_0}\}=\bigcap\{\overline{f_1(U\cap D)}: U\in \mathcal{U}_{x_0}\}=\overline{f}_1(x_0)$$
$$ =\Phi(x_0).$$
Let $x_0\in L_2$. Using Lemma~\ref{L2} for a directed set ${(M,\leq)=(\mathcal{U}_{x_0},\supseteq)}$, decreasing families $A_U=\overline{f(U\cap D_1)}$ and ${B_U=\overline{f(U\cap D_2)}}$, where $U\in \mathcal{U}_{x_0}$, we have that
$$ \overline{f}(x_0)=\bigcap\{\overline{f(U\cap D)}: U\in \mathcal{U}_{x_0}\}=\bigcap\{\overline{f(U\cap D_1)}\cup \overline{f(U\cap D_2)}: U\in \mathcal{U}_{x_0}\}$$
$$\subseteq \bigcap\{\overline{f(U\cap D_1)}: U\in \mathcal{U}_{x_0}\}\cup \bigcap\{\overline{f(U\cap D_2)}: U\in \mathcal{U}_{x_0}\}$$
$$ =\bigcap\{\overline{f_1(U\cap D_1)}: U\in \mathcal{U}_{x_0}\}\cup \bigcap\{\overline{f_2(U\cap D_2)}: U\in \mathcal{U}_{x_0}\}$$
$$ =\overline{f}_1 (x_0)\cup \overline{f}_2 (x_0)\subseteq \Phi (x_0).$$
\end{proof}

\end{document}